\documentclass[12pt]{amsart}
\usepackage{geometry}                
\usepackage[parfill]{parskip}    
\usepackage{graphicx}
\usepackage{amssymb}
\usepackage{amsmath}
\usepackage{amsthm}
\usepackage[foot]{amsaddr}
\usepackage{epstopdf}
\usepackage{extarrows}
\usepackage{color}
\usepackage[all,cmtip]{xy}
\usepackage{url}

\newtheorem{thm}{Theorem}[section]
\newtheorem{cor}[thm]{Corollary}

\newtheorem{prop}[thm]{Proposition}

\theoremstyle{definition}
\newtheorem{defin}[thm]{Definition}
\theoremstyle{remark}
\newtheorem{rem}[thm]{Remark}

\numberwithin{equation}{section}

\newcommand{\R}{\mathbb{R}}

\newcommand{\K}{\mathbb{K}}

\newcommand{\cA}{\mathcal{A}}
\newcommand{\cC}{\mathcal{C}}
\newcommand{\cD}{\mathcal{D}}

\newcommand{\cJ}{\mathcal{J}}
\newcommand{\cM}{\mathcal{M}}
\newcommand{\cH}{\mathcal{H}}
\newcommand{\cG}{\mathcal{G}}

\newcommand{\fD}{\mathfrak{D}}
\newcommand{\fd}{\mathfrak{d}}
\newcommand{\fs}{\mathfrak{s}}

\newcommand{\gG}{\Gamma}
\newcommand{\dginf}{D^{G^{\infty}}}
\newcommand{\sinf}{S^{G^{\infty}}}
\newcommand{\mc}{\mathcal}
\newcommand{\plim}{\text{Lim\text{ }}}
\newcommand{\colim}{\text{Colim\text{ }}}

\begin{document}

\title{On Enriched Categories and Induced Representations}
\author[Leslie]{Joshua A. Leslie $^1$}
\email[]{joshualeslie1@mac.com}
\address{$^1$ Mathematics Department, Howard University, Washington DC 20059, USA}
\author[Twum]{Ralph A. Twum $^2$}
\email[]{ratwum@ug.edu.gh}
\address{$^2$ Department of Mathematics, University of Ghana, Legon, Ghana}


\begin{abstract}
We show that induced representations for a pair of \emph{diffeological Lie groups} exist, in the form of an indexed colimit in the category of diffeological spaces. 
\end{abstract}
\maketitle

\section{Introduction}
Let $G$ be a group. A representation of $G$ may be described as a functor 
$F : G \rightarrow Vect_{\K}$, where $G$ is a single-set category with the elements of $G$ forming the arrows of the category, and $Vect_{\K}$ is the category of vector spaces over a field $\K$. 
Any subgroup $H$ of $G$ is automatically a full subcategory of $G$, hence the representation of
$G$ restricts to a representation of $H$. Thus the functor $F$ restricts to a subfunctor on $H$, which 
we may also denote by $F$.

Now, given a functor $T : H \rightarrow Vect_{\K}$, we consider if $T$ can be \emph{canonically}
extended to a functor from $G$ to $Vect_{\K}$. This process of \emph{induction} has been defined
for a variety of pairs $G, H$:
\begin{itemize}
\item[(a)] $G, H$ finite groups; 
\item[(b)] $G, H$ complex analytic groups, with $H$ a closed subgroup of $G$. This led to the 
Borel-Weil theorem and its extensions by Bott and Kostant(\cite{Bott}, \cite{Kos}, \cite{Dem}, \cite{Lur}). Other cases have also been considered, for instance
\item[(c)] $G, H$ direct limit groups, with $H$ a closed subgroup of $G$. In this case both $G$ and $H$ are infinite dimensional (\cite{Kum}, \cite{DP}).
\end{itemize}

We want to consider the concept of induced representations of infinite dimensional Lie groups using results from enriched category 
theory and diffeology. We will specifically deal with the pair of groups $G = \dginf(M)$ and $H = \sinf(M)$ defined below. 

Our goal is the nonconstructive proof of the following theorem:
\begin{thm}\label{mainthm} Let $\sinf(M)$ be the regular diffeological Lie group of super symplectic diffeomorphisms of a supermanifold $M$. 
$\sinf(M)$is a Lie subgroup of the group of super diffeomorphisms $\dginf(M)$; further, there exist induced representations for Hilbert space representations for these two regular Lie groups.
\end{thm}

We use existing results from enriched category theory to show that the induced representation exists, first as a Kan extension, which can be written as an \emph{indexed limit}. In order to show that the representation exists, we need to enlarge the category from that of smooth (super-)manifolds to a more convenient category. We use the category of \emph{diffeological spaces}, whose properties we will recall in the following section.  

This paper is organized as follows. In Section 2, we recall some useful facts about Diffeological spaces, and discuss the category of diffeological spaces, which include (super) manifolds. We also define the two groups in Theorem \ref{mainthm} above.

Sections 3 and 4 introduce the main categorical notions: Kan extensions and indexed limits, and we show that induced representations can be interpreted as Kan extensions, which can also be interpreted as indexed limits. We collect these facts and prove Theorem \ref{mainthm} in Section 5. Section 6 contains some remarks on Diffeological categories and regularity of (diffeological) Lie groups. 
\section{Diffeology and Diffeological Spaces}

\subsection{Diffeological Spaces}
The axioms of a diffeological space, or a diffeology on a set, were set forth by Iglesias-Zemmour \cite{Igl}. In the definition of a smooth manifold, we define a maximal atlas consisting of compatible charts on a set $X$. In a diffeology, we focus on the maps themselves between coordinate patches of open sets in $\R^{n}, n\ge0$ and the set $X$, satisfying three axioms. These maps are called \emph{plots} of the diffeology, and the set $X$ with the defined collection of plots is called a diffeological space. The advantage of considering a diffelogical space is that smooth manifolds have a natural diffeology, and so form a full subcategory of the category of diffeological spaces.

\begin{defin}
Let $X$ be a set. An $n$-parametrization of $X$ is defined as a map $f : U \rightarrow X$, where $U$ is an open set in $\R^{n}, n\ge 0$. A constant parametrization is a map of the form $f : U\rightarrow X$ such that $f(u) = x$ for all $u \in U$.

\end{defin}

\begin{defin}
A diffeology $\cD$ on $X$ is a collection of parametrizations of $X$ satisfying the following axioms:

\begin{enumerate}
\item[(D1)] $\cD$ contains the constant parametrizations.
\item[(D2)] If $p: U \rightarrow X$ is a parametrization such that $\forall u \in U$, there exists a neighborhood $V$ of $u$ such that $p|_{V}$ is in $\cD$, then $p$ itself is in $\cD$. In other words, the collection $\cD$ satisfies the axioms of a set-valued sheaf.
\item[(D3)] Given a parametrization $p:U \rightarrow X$ in $\cD$, if $f: V \rightarrow U$ is a smooth map, then $p \circ f$ is a parametrization in $\cD$.
\end{enumerate}
\end{defin}

The parametrizations in $\cD$ are called \emph{plots} of the diffeology on $X$ and the pair $(X,\cD)$ is called a \emph{diffeological space}. 

\begin{defin}
Let $X$ be a set and $\mathcal{F}$ be a collection of parametrizations on $X$. The diffeology on $X$ induced by $\mathcal{F}$ is the intersection of all diffeologies on $X$ containing $\mathcal{F}$. 
\end{defin}

\begin{defin}
Let $(X,\mathcal {F})$ and $(Y,\mathcal{G})$ be diffeological spaces. A map $f:X\rightarrow Y$ is \emph{smooth} if $f$ takes plots of the diffeology on $X$ to plots of the diffeology on $Y$.
\end{defin}

A diffeological group is defined naturally, as a group such that the product and inverse maps are smooth in the sense of diffeology. Since Lie groups (resp. super Lie groups) are manifolds (resp. supermanifolds), they are diffeological groups. 

\begin{defin}
 A difeological space $(S,F)$ is \emph{lattice type} (or \emph{L-type}) when given any two plots $f:M_1 
 \rightarrow S$, $g:M_2 \rightarrow S$ 
 at a point $f(x) = g(y)$, there exists a third plot $h$ through which the germs of $f$ at $x$ and $g$ 
 at $y$ factor, that is there exists a map $h:N 
 \rightarrow S$ with $f = h\circ \varphi$, $g=h\circ \gamma$ where $\varphi: \tilde{M_1} \rightarrow N
 $ and $\gamma: \tilde{M_2} 
 \rightarrow N$ are plots such that $\tilde{M_i} \subset M_i$, $i=1,2$ are neigborhoods of $x$ and $y$ 
 respectively and $\varphi(x) = \gamma(y)$.
 \end{defin}
 
 All diffeological groups are $L$-type (\cite{Lau}, Lemma $2.37$). The property of being $L$-type allows for the definition of the tangent space at the identity, and a diffeological structure on that space, which gives a diffeological vector space. 
\begin{defin}
A diffeological group $G$ is called a \emph{diffeological Lie group} when $T_{e}G$ is a diffeological vector space such that 
\begin{enumerate}
\item [(DL1)] For any nonzero vector $\alpha \in T_{e}G$, there exists a smooth real-valued map $T:T_{e}G \rightarrow \R$ such that $T(\alpha) \ne 0$:
\item [(DL2)] Every plot of $T_{e}G$ factors smoothly through a smooth linear map of a compact Hausdorff topological vector space into $T_{e}G$.
\end{enumerate}
\end{defin}

Let us apply the above definitions to that of supermanifolds. Using the construction in the paper by Leslie \cite{Les1}, we form a super vector space of super dimension $m,n$. We use this to define the open sets for the  diffeology on super manifolds and super Lie groups. The parametrizations (and hence the plots) can be adapted similarly.

Let $\Gamma$ be the Grassmann algebra over $\R$ generated by the elements $1$ and $\{\xi\}_{i\in I}$, where $I$ is a countably infinite set. In line with the theory of super vector spaces, all $\xi_{i}$ are odd, with degree 1. Let $\mathcal{J}$ be the collection of finite subsets of $I$, ordered by inclusion. Each homogeneous element of $\Gamma$ may be expressed as $x_{K}\xi_{i_1}\xi_{i_{2}}\dots \xi_{i_{n_{K}}}$, $x_{K} \in \R$, where $K \in \cJ$ and $n_{K}$ is the cardinality of $K$. We will use the notation in Rogers \cite{Rog}, in which even homogeneous elements are represented using latin letters and odd homogeneous elements are represented using Greek letters. 

As seen in \cite{Les1}, $\Gamma$ is a super commutative algebra, with $ab = (-1)^{|a||b|}ba$, where $a,b$ are elements of $\Gamma$, with the degree of elements of $\Gamma$ extended by linearity from the degree of homogeneous elements of $\Gamma$. 

Define $V$ to be the $m,n$-dimensional superspace $\Gamma_{0}^{m} \oplus \Gamma_{1}^{n}$, where $\Gamma_{0}$ is the subspace of $\Gamma$ consisting of even elements and $\Gamma_{1}$ is the subspace of $\Gamma$ consisting of odd elements. Thus 
\[
V = \Gamma_{0} \times \Gamma_{0}\times \cdots \times \Gamma_{0} \times \Gamma_{1} \times \Gamma_{1} \cdots \times \Gamma_{1}\quad (m,n)\,\text{times}.
\]

We define a topology on $V$ as follows: define the inductive limit topology on $\Gamma$, that is the finest topology such that all inclusion maps $\Gamma_{K} \hookrightarrow \Gamma$ are continuous, where $\Gamma_{K}$ is the finite dimensional subspace generated by elements $\xi_{i_{1}},\dots,\xi_{i_{n_{K}}}$, $K \in \cJ$. The topologies on $\Gamma_{0}$ and $\Gamma_{1}$ are defined by restricting the inductive limit topology to the even and odd subspaces of $\Gamma$ respectively. Finally the topology on $V$ is defined using the product topology on $\Gamma_{0}^{m}$ and $\Gamma_{1}^{n}$.

For the objects under study, we have modified Iglesias' definition of a diffeological space as follows: instead of using open sets in $n$-dimensional space $\R^{n}$ where $n\ge 1$, we consider open subsets of the spaces $\gG_{0}^{m} \oplus \gG_{1}^{n}$, where $m,n \ge 0$.

We use the following definition of a  supersmooth $(G^{\infty})$ function from \cite{Les2}:

\begin{defin}
Let $V$ and $W$ be topological graded modules over $\gG_{0}$. A continuous map $f : V \times \cdots \times V \rightarrow W$ is said to be an $n$-multimorphism if 
\begin{enumerate}
\item[(M1)] $f$ is $n$-multilinear with respect to the ground field $\R$,
\item[(M2)] $f(v_{1}, \dots v_{i}\gamma,v_{i+1}, \dots v_{n}) = f(v_{1}, \dots, v_{i}, \gamma v_{i+1}, \dots, v_{n})$, where $v_{i} \in V, 1\le i \le n$ and $\gamma \in \gG_{0}$,
\item[(M3)] $f(v_{1}\dots v_{n}\gamma) = f(v_{1},\dots, v_{n})\gamma$, $v_{i}\in V$, $1 \le i \le n$, $\gamma \in \gG_{0}$.
\end{enumerate}
\end{defin}

\begin{defin}
Let $V$ and $W$ be topological graded modules over $\gG_{0}$ and let $U\subseteq V$ be open in $V$. A function $f: U \rightarrow W$ will be called supersmoooth or $G^{\infty}$ when for every $n\ge 1$, there exist continuous maps which are regular $k$-multimorphisms in the k terminal variables for each fixed $x \in U$, denoted by $D^{k}f(x;\cdots): U \times V \times \cdots \times V \rightarrow W$, where $k \le n$, such that the map
\[
F_{k}(h):=f(x+h)-f(x)-Df(x;h)-\frac{1}{2!}D^{2}f(x;h,h) - \cdots - \frac{1}{k!}D^{k}f(x;h,\dots h), 1\le k \le n
\] 
satisfies the property that

\[
G_{k}(t,h) := \left\{\begin{array}{cc}
\frac{F_{k}(th)}{t^{k}}, & t \ne 0 \\
0, & t = 0 \end{array}\right.
\]
is continuous at $(0,h)$ in $V$. 
\end{defin} 



Now, let $M$ be a supermanifold modeled on $V = \gG_{0}^{m} \oplus \gG_{1}^{n}$, where $m$ and $n$ are fixed. Thus the charts are $G^{\infty}$ maps from open subsets $U$ of $V$ into $M$. Then $M$ is automatically a diffeological space, with the diffeology generated by the $(m,n)$-charts on $M$. This diffeology is by definition the smallest diffeology containing all such charts as plots.

\subsection{The category of Diffeological Spaces}
Let us denote the category of diffeological spaces by $\fD$. By definition, $\fD$ 
contains the category of smooth manifolds (and supermanifolds) as a full subcategory. 
It also has several useful properties outlined below, as shown by Iglesias \cite{Igl}. 

\begin{prop}
$\fD$ is closed under coproducts. For a collection $\{X_{i}|i\in I\}$ of 
diffeological spaces, the coproduct of the collection is called the \emph{sum} of the diffeological spaces $X_{i}$ and is denoted $\displaystyle \sqcup_{i\in I}X_{i}$ (Iglesias \cite{Igl}, 1.39). 
\end{prop}

\begin{prop}
$\fD$ is closed under products. Just as in the above result, given an arbitrary collection $\{X_{i}|i\in I\}$, the product object is denoted $\displaystyle \prod_{i\in I}X_{i}$ and can be given a canonical diffeology (Iglesias \cite{Igl}, 1.55).
\end{prop}

\begin{prop}
$\fD$ is a symmetric monoidal category. 
\end{prop}
\begin{proof}
The product map is derived from the definition of the product object. In addition, since $\fD$ is essentially a concrete category, associativity of the product map holds. We can take any one-point set as a unit object, since all one-point sets are unique up to isomorphism of diffeological spaces (all plots are constant). The coherence axioms are clearly satisfied. The symmetry of $\fD$ is given by the canonical map $\gamma : X \times Y \rightarrow Y \times X, (x,y)\mapsto (y,x)$. 
\end{proof}

\begin{prop}
$\fD$ is complete with respect to (all small) limits and colimits. 
\end{prop}
\begin{proof}
From Mac Lane (\cite{Mac}, Corollary V.2.2), we know that 
if a category $\cC$ has equalizers of all pairs of arrows and all small products, then $\cC$ is small-complete. In $\fD$, the equalizer of pairs of smooth maps $f,g: X \rightarrow Y$ is defined to be the subspace $Z$ of $X$ such that $f(x) = g(x)$ for $x\in Z$. The universal arrow is the canonical inclusion map. The coequalizer is defined for pairs $f,g : X \rightarrow Y$ by considering the smallest equivalence relation $E \subset Y \times Y$ containing all pairs $(f(x), g(x))$, for $x \in X$. The coequalizer is the quotient space $Z = Y/E$, with the universal arrow given by the canonical projection map. From the existence of equalizers, coequalizers, coproducts and products, we conclude that $\fD$ is complete. 
\end{proof}

\begin{prop}
$\fD$ is Cartesian closed. 
\end{prop}

\begin{proof}
The map $\phi :\fD(X\times Y, Z) \rightarrow \fD(X, \fD(Y,Z))$ given by $(\phi(f)(x))(y) = f(x,y)$ for $f: X\times Y \rightarrow Z$ is a bijection, and shows that the functor $\fD(Y,-): \fD \rightarrow \fD$ is the right adjoint to the functor $- \times Y:\fD \rightarrow \fD$.
\end{proof}

\begin{rem}
The properties of $\fD$ as a category make it suitable as an enriching category. This gives us the ability to use the results of Kelly \cite{Kel} on indexed limits, indexed colimits and Kan extensions in the case of enriched categories. These ideas are discussed in the following section. 
\end{rem}

Now we can define the (super) Lie groups we make use of in the paper. 
\subsection{The super Lie groups $\dginf(M)$ and $\sinf(M)$}

\begin{defin}
$\dginf(M)$ is defined to be the group of super diffeomorphisms of $M$, where $M$ is a supermanifold. $\dginf(M)$ can be made into a diffeological space by applying the \emph{function space diffeology:} a map $f: U \subseteq \dginf(M)$ is a plot if the induced map $F : U \times M \rightarrow M$ given by $F(u,m) = (f(u))(m)$ is smooth for all $u \in U$. The smallest diffeology generated by such plots is called the function space diffeology on $\dginf(M)$. 

\end{defin}

Finally, let us consider the space $\sinf(M)$ of super symplectic diffeomorphisms of $M$, defined as follows.

In \cite{Les2}, the following vector bundles on $M$ were defined: 

\begin{itemize}

\item $T_{\Gamma}M$, which is the associated vector bundle with fiber $\Gamma \otimes_{\Gamma_{0}} V$, where $V = \gG_{0}^{m} \oplus \gG_{1}^{n}$ as defined earlier. 

\item $T_{\Gamma}^{\ast} M := (T_{\Gamma}M)^{\ast}$, which is the associated vector bundle with fiber the dual of the space $\Gamma \otimes_{\Gamma_{0}} V$, $Hom(\Gamma \otimes_{\Gamma_{0}} V, \Gamma)$. 
\end{itemize}

Define the bundle $\pi :\wedge^{p} T^{\ast}_{\Gamma}M
\otimes_{\Gamma_{0}}\otimes^{q}_{\Gamma_{0}}T_{\Gamma}M \rightarrow M$ \cite{Les2}. Let $\omega$ be a non zero section of the above bundle  and define the Lie algebra $\fs$ to be the set $\{ X \in \fd\,|\, L_{x}(w) = 0\}$. Suppose further that $\fs$ is a Lie subalgebra of $\fd$, where $\fd$ is the Lie algebra of the group of superdiffeomorphisms of $M$. Leslie \cite{Les2} has shown that $\fs$ is a strongly integrable diffeological Lie subalgebra of $\fd$, which indicates that there exists a regular diffeological Lie group such that the exponential map is smooth. We shall call this Lie group $\sinf(M)$. $\sinf(M)$ is by definition a subgroup of $\dginf(M)$.

\section{Induced Representations are Kan Extensions}

Let $\cA$, $\cC$ and $\cM$ be small categories. Recall that (\cite{Mac}, Chapter $X$) a \emph{right Kan extension of a functor $T: \cM \rightarrow \cA$ along $K:\cM \rightarrow \cC$} is a functor $R = Ran_{K}T : \mc{C} \rightarrow \mc{A}$ and a natural transformation $\mu: RK \dot{\longrightarrow} T$ such that for any other right extension $S$ of $T$ along $K$ and $\sigma: SK \dot{\longrightarrow} T$, there exists a unique natural transformation $\varepsilon: S \dot{\longrightarrow} R$ such that $\sigma = \mu \circ \varepsilon K:
SK \dot{\longrightarrow} RK \dot{\longrightarrow} T$. This gives the following bijection of sets:
\begin{equation}\label{adjunction1}
\mc{A^{C}}(S, Ran_{K}T) \simeq \mc{A^{M}} (SK, T)
\end{equation}
This bijection is an adjunction between the functor categories $\mc{A^{C}}$ and $\mc{A^{M}}$. By the universal property, $Ran_{K}T$ is unique up to natural isomorphism.

Dually, a \emph{left Kan extension of $T$ along $K$} is a functor $L = Lan_{K}T: \mc{C} \dot{\rightarrow} \mc{A}$ and a natural transformation $\mu : T \dot{\longrightarrow} LK$ such that for any other left extension $S$ of $T$ along $K$ and $\sigma: T \dot{\longrightarrow} SK$, there exists a unique natural transformation $\varepsilon: L \dot{\longrightarrow} S$ such that $\sigma= \varepsilon K \circ \mu : T \dot{\longrightarrow} LK \dot{\longrightarrow} SK$. We thus have the following adjunction
\begin{equation}\label{adjunction2}
\mc{A^{C}}(Lan_{K}T,S) \simeq \mc{A^{M}}(T, SK).
\end{equation}

If $\mc{A}$ is a small complete category, then (\cite{Mac}, Chapter $X$) the Kan extensions can be expressed as pointwise limits: for the functors $T: \mc{M} \rightarrow \mc{A}$ and 
$K: \mc{M} \rightarrow \mc{C}$, the left and right Kan extensions are functors $L,R: \mc{C} \rightarrow \mc{A}$, where for each object $c$ of $\mc{C}$, 

\[
	\begin{aligned}
	Lc &= \colim \left[(K \downarrow c) \xlongrightarrow{Q} \mc{M} \xlongrightarrow{T} \mc{A}  \right] = \colim_{f}Tm, \qquad f \text{ in } (K\downarrow c) \\
	Rc &= \plim \left[(c \downarrow K) \xlongrightarrow{Q} \mc{M} \xlongrightarrow{T} \mc{A}  \right] = \plim_{f}Tm, \qquad f \text{ in } (c\downarrow K)\\
	\end{aligned}
\]

where $(K\downarrow c)$ and $(c\downarrow K)$ are the comma (or slice) categories of objects over and under $c$, respectively. Now, we can rewrite $(K\downarrow c)$ and $(c\downarrow K)$ as follows: consider the functor $F = \mc{C}(c, K-): \mc{M} \rightarrow \text{Set}$, which assigns to each object of $\mc{M}$ the set of arrows of the form $f: c \rightarrow Km$, where $m$ is an object of $\mc{M}$. The category of \emph{elements of $F$} has objects pairs $<m, f>$, where $m$ is an object of $\mc{M}$ and $f \in \mc{C}(c, Km)$. If $\tau: m \rightarrow m^{\prime}$ is an arrow in $\mc{M}$, then we have an induced arrow $\tau^{\ast}: <m,f> \rightarrow <m^{\prime}, f^{\prime}>$, given by the commutative diagram

\begin{displaymath}
\xymatrix{&c\ar[dl]_{f}\ar[dr]^{f^{\prime}} \\
Km\ar[rr]^{K\tau}& &Km^{\prime}
}
\end{displaymath}

Therefore the slice category $(c\downarrow K)$ is precisely the category of elements of $F = \mc{C}(c, K-)$. Similarly, $(K \downarrow c)$ is the category of elements of $\mc{C}(K-, c)$.

Let us relate the above definitions to the problem of induced representations, using the categories
$H$, $G$ and $\mc{V} = Vect_{\K}$ which were defined in the Introduction, in place of $\mc{A}$, $\mc{C}$ and $\mc{M}$ above. We notice immediately that the induced representation can be interpreted as a Kan extension of the functor $T: H \rightarrow \mc{V}$. The left and right Kan 
extensions of $T $ lead in principle to two formulations of the induced
representation of $T$, where it exists. In the finite case, the two induced representations, denoted $Ind$ and $coInd$ exist and are isomorphic. 

We will focus on the left Kan extension in our case. We will prove that the left Kan extension does exist, hence prove  Theorem \ref{mainthm}, using the tools provided by diffeology. In the next section, we discuss indexed limits, and illustrate how the Kan extension can be expressed as an indexed limit.

\section{Indexed Limits in Diffeological Categories}
The properties of the category $\fD$ of diffeological spaces enable us to define an \emph{enriched category} with the morphisms between objects having extra structure:
\begin{defin} (\cite{Twu} Definition 5.1.1)
 A \emph{Diffeological Category} $\mc{D}$ is a locally small category such that for any objects $x,y$ in $\mc{D}$, the hom set $\mc{D}(x,y)$ is a 
 diffeological space which satisfies the following axiom: for all objects $x,y$ and $z$ in $\mc{D}$ and elements $g$ in $\mc{D}(x,y)$ and $f$ in $\mc{D}(y,z)$, the composition $f\circ g$ in $\mc{D}(x,z)$ defines a smooth map:
 \begin{equation}\label{diffcat}
 \begin{aligned}
 \mc{D}(y,z) \times \mc{D}(x,y) &\rightarrow \mc{D}(x,z)\\
 f \times g &\mapsto f \circ g \\
 \end{aligned}
 \end{equation}
 where $\mc{D}(y,z) \times \mc{D}(x,y)$ has the product diffeology. 
 \end{defin}
 
Thus a diffeological category is an enriched category where arrows between objects have the structure of a diffeological space.
Functors defined between diffeological categories are called \emph{smooth} when they induce maps between diffeological spaces of arrows. Smooth natural transformations, smooth adjunctions and smooth universal arrows are defined similarly. 

We note that the concept of a diffeological category makes sense, since the category $\fD$ of diffeological spaces is itself a diffeological category: the function space diffeology makes the set of functions between diffeological spaces into a diffeological space itself.

Let $G: \cJ \rightarrow \cC$ be a functor of diffeological categories. We observe, from \cite{Kel} and \cite{Fre} that the limits and colimits of $G$ satisfy the following natural isomorphisms:

\[
\begin{aligned}
\cC(c, \plim G) &\simeq \cC^{\cJ}(\Delta c, G)\\
\cC(\colim G, c) &\simeq \cC^{\cJ}(G, \Delta c)\\ 
\end{aligned}
\]

We can rewrite the cones $\cC^{\cJ}(\Delta c, G)$ and $\cC^{\cJ}(G, \Delta c)$ as
$\fD^{\cJ}(\Delta 1, \cC(c, G-))$ and $\fD^{\cJ^{op}}(\Delta 1,\cC(G-,c))$ respectively, where in the second case, $\Delta 1$ is regarded as a contravariant functor from $\cJ$ to $\cC$. 

The idea of an indexed (or weighted) limit is to replace the functor $\Delta 1$ by an index functor $F: \cJ \rightarrow \cC$.

\begin{defin}\label{indexformulas}
Let $F:\cJ \rightarrow \fD$ and $G:\cJ \rightarrow \cC$ be smooth functors of diffeological categories. Fix an object $c$ of $\mc{C}$. \emph{An $(F,c)$-cylinder over $G$} is a smooth natural transformation $\alpha: F \dot{\rightarrow}\cC(c,G-)$. It can be illustrated with the commutative diagram:
\begin{displaymath}
\xymatrixcolsep{5pc}\xymatrix{Fi \ar[r]^{\alpha_{i}}\ar[d]_{f}& \mc{C}(c,Gi)\ar[d]^{\mc{C}(c,Gf)}\\
Fj\ar[r]_{\alpha_{j}}&\mc{C}(c,Gj)} 
\end{displaymath}
\end{defin}
As discussed above, if $F=\Delta1 : \cJ \rightarrow \fD$, the functor assigning to each object $i \in \cJ$ the unit element $\ast$, the $(\Delta 1, c)$-cylinder can be identified with a cone $\tau: \Delta c \dot{\rightarrow} G$.

We define the \emph{limit of $G$ indexed by $F$} to be the object $\{F,G\}$ in $\cC$ such that every other $(F,c)$-cylinder
over $G$ factors uniquely though $\{F,G\}$. The definition is illustrated by the following commutative diagram:
\begin{displaymath}
\xymatrixcolsep{5pc}\xymatrix{Fi \ar[r]^{\alpha_{i}}\ar@{}[d]_{=}& \cC(c,Gi)\\
Fi\ar[r]^{\mu_{i}}&\cC(\{F,G\},Gi)\ar@{-->}[u]_{\cC(t,1)}} 
\end{displaymath}
where $t: c \rightarrow \{F,G\}$ is the unique arrow. Hence $\alpha_{i} = \cC(t,1)\mu_{i}$.
This gives the isomorphism (in our context a diffeomorphism of diffeological spaces)
\begin{equation}\label{eq:limit}
\cC(c, \{F,G\}) \simeq \fD^{\cJ}(F, \cC(c,G-))
\end{equation}
The universal arrow is $\mu: F \dot{\rightarrow} \cC(\{F,G\},G-)$. 

\begin{equation}\label{canlimit}
\{\Delta1, G\} \simeq  \plim G,
\end{equation}
the usual limit of $G: \cJ \rightarrow \cC$.

\begin{defin}
Let $F:\cJ^{op} \rightarrow \fD$ and $G: \cJ \rightarrow \cC$ be smooth functors of diffeological categories. \emph{An $(F,c)$-cylinder under G} is a natural transformation of smooth functors $\alpha: F \dot{\rightarrow} \cC(G-,c)$. Again, in the case $F=\Delta 1$, $\alpha: \Delta1 \dot{\rightarrow} \mc{C}(G-,c)$ can be identified with a cone from the base $G$ to $c$. The \emph{colimit of $G$ indexed by $F$} is the unique (up to isomorphism) object of $\mc{C}$ such that every $(F,c)$- cylinder under $G$ factors uniquely through the universal cylinder $\mu: F\dot{\rightarrow} \cC(G-,F*G)$. The commutative diagram illustrating the definition is given below.

\begin{displaymath}
\xymatrixcolsep{5pc}\xymatrix{Fi \ar[r]^{\alpha_{i}}\ar@{}[d]_{=}& \cC(Gi,c)\\
Fi\ar[r]^{\mu_{i}} &\cC(Gi,F*G)\ar@{-->}[u]_{\cC(1,t)}} 
\end{displaymath}

where $t: F\ast G \rightarrow c$ is the unique map. Hence $\alpha_{i} = \cC(1,t)\mu_{i}$.
\end{defin}
We have the following diffeomorphism of diffeological spaces:
\begin{equation}\label{eq:colimit}
\cC(F*G,c) \simeq \fD^{\cJ^{op}}(F,\cC(G-,c))
\end{equation}
As before, 
\begin{equation}\label{cancolimit}
\Delta 1*G \simeq \colim\,G.
\end{equation}

When we have the smooth functors $F,G: \mc{J} \rightarrow \fD$, the indexed limits and colimits take a particularly simple form:

In some cases, we can express indexed limits as ordinary limits. Consider a smooth functor $F: \mc{C} \rightarrow \fD$. F can be expressed as a colimit of representable functors (\cite{Mac}, Theorem $III.7.1$):

\begin{equation}\label{repcolim}
F \simeq Colim(el F \xlongrightarrow{Q} \mc{C}\xlongrightarrow{M} \fD^{\mc{C}} \xlongrightarrow{y^{-1}} \fD),
\end{equation}

where $Q$ is the projection functor $(c,x) \mapsto c$, $M$ is the functor which takes objects $c$ in $\mc{C}$ to hom-objects $\mc{C}(c,-)$ in $\fD$, and $y^{-1}$ is induced by the Yoneda lemma. We can extend this result as follows:

\begin{prop}\label{natend} Let $F,G : \mc{J} \rightarrow \fD$ be smooth functors. Then $\{
F,G\} \simeq Nat(F,G)$, the diffeological space of natural transformations between the functors $F$ and $G$. 
\end{prop}

\begin{proof}
We use the fact that the space of natural transformations between $F$ and $G$ can be expressed as an \emph{end}: $Nat(F,G) \simeq \int_c \fD(Fc, Gc)$. Using the language of ends, we obtain the following isomorphisms:

\[
\begin{aligned}
\fD(c, \{F,G\}) & \simeq \fD^{\mc{J}}(F, \fD(c,G-)) = Nat(F, D(c, G-))\\
 	&\simeq \int_x \fD(Fx, D(c,Gx)) \\
	& \simeq \fD\left( c, \int_x \fD(Fx, Gx)\right) = \fD(c, Nat(F,G))\\
\end{aligned} 
\]
Therefore $\{F,G\}$ is isomorphic to $Nat(F,G)$. 
\end{proof}

Using the same preamble as the previous proposition, we can express the indexed limits and indexed colimits as diffeological spaces of natural transformations. We will state two propositions that are needed in the sequel. 

\begin{prop}Let $F: \mc{J} \rightarrow \fD$ and $G: \mc{J} \rightarrow \mc{C}$ be smooth functors, where $\mc{C}$ is a diffeological category. For any object $c$ of $\mc{C}$, we have the following isomorphism of diffeological spaces:
\begin{equation}
\mc{C}(c, \{F,G\}) \simeq \{F,\mc{C}(c, G-)\}.
\end{equation}
\end{prop} 

\begin{proof}
$\mc{C}(c,G-)$ is a smooth functor from $\mc{J}$ to $\fD$, so applying the previous proposition (\ref{natend}) and the definition of indexed limit \eqref{eq:limit}, we have
\begin{displaymath}
\begin{aligned}
\mc{C}(c, \{F,G\}) &\simeq \fD^{\mc{J}}(F, \mc{C}(c,G-))\\
&= Nat(F,\mc{C}(c,G-))\\
&\simeq \{F,C(c,G-)\}\\
\end{aligned}
\end{displaymath}
\end{proof}

\begin{prop}Let $F \mc{J}^{op} \rightarrow \fD$  and $G: \mc{J} \rightarrow \mc{C}$ be smooth functors, where $\mc{C}$ is a diffeological category. For any object $c$ of $\mc{C}$, we have the following isomorphism of diffeological spaces:
\begin{equation}
\mc{C}(F\ast G, c) \simeq \{F,\mc{C}(G-, c)\}.
\end{equation}
\end{prop} 

\begin{proof}
$\mc{C}(G-,c):\mc{J}^{op} \rightarrow \fD$ is a smooth functor, so applying the definition of indexed colimits and Proposition \ref{natend} we get
\begin{displaymath}
\mc{C}(F*G,c) \simeq \fD^{\mc{J}^{op}}(F, \mc{C}(G-,c)) = Nat(F, \mc{C}(G-,c)) \simeq \{F,\mc{C}(G-,c)\}
\end{displaymath}
\end{proof}

\begin{thm}
Let $F: \mc{C} \rightarrow \fD$ and $G: \mc{C} \rightarrow \fD$ be smooth functors. Then $Nat(F,G) \simeq \plim(elF\xrightarrow{Q}\mc{C} \xrightarrow{G} \fD)$ as diffeological spaces.
\end{thm}

\begin{proof}
The limiting cone for the composite functor $elF\xrightarrow{Q}\mc{C} \xrightarrow{G} \fD$ consists of objects $\alpha_{(c,x)}: \star \rightarrow Gc$, where $\star$ denotes a 1-point set. For a map $f^{*}: (c,x) \rightarrow (c^{\prime},x^{\prime})$, $Gf \alpha_{(c,x)} = \alpha_{(c^{\prime},x^{\prime})}$, since $\alpha$ is a cone. Define a natural transformation $\beta: F\dot{\rightarrow}G $ by $\beta_{c}(x) := \alpha_{(c,x)}$. For $f^{*}: (c,x) \rightarrow (c^{\prime},x^{\prime})$, we have:
\begin{displaymath}
\begin{aligned}
Gf(\beta_{c}(x)) &= Gf\alpha_{(c,x)}= \alpha_{(c^{\prime},x^{\prime})}\\
\beta_{c^{\prime}}(Ff(x)) &= \beta_{c^{\prime}}(x^{\prime}) = \alpha_{(c^{\prime},x^{\prime})}
\end{aligned}
\end{displaymath} giving the following
commutative diagram:
\begin{displaymath}
\xymatrixcolsep{5pc}\xymatrix{
Fc\ar[r]^{\beta_{c}}\ar[d]_{Ff} & Gc\ar[d]^{Gf}\\
Fc^{\prime}\ar[r]_{\beta_{c^{\prime}}} & Gc^{\prime}
}
\end{displaymath}
Hence $Nat(F,G) \simeq \plim(elF\xrightarrow{Q}\mc{C} \xrightarrow{G} \fD) \simeq \{F,G\}$.
\end{proof}

\begin{cor}\label{replim1}
Let $G: \mc{C} \rightarrow \mc{B}$ be a smooth functor of diffeological categories. Then given a smooth index functor $F: \mc{C} \rightarrow  \fD$,
\begin{displaymath}
\{F,G\} \simeq \plim(elF \xrightarrow{Q} \mc{C} \xrightarrow{G} \mc{B})
\end{displaymath}
\end{cor}

\begin{proof}
Equation \eqref{repcolim} gives $F \simeq \colim \left(elF \xrightarrow{Q} \mc{C} \xrightarrow{y^{-1}} \fD^{\mc{C}}\right) \simeq \colim \left(C(c,-): el F \rightarrow \fD^{\mc{C}}\right)$. Therefore:
\begin{displaymath}
\begin{aligned}
\{F,G\} \simeq \{\colim \mc{C}(c,-),G\} & \simeq \{\Delta 1*\mc{C}(c,-),G\}\\
&\simeq \left\{\Delta 1,\{\mc{C}(c,-), G\}\right\}\\
&\simeq \{\Delta 1,Gc\}\\
&\simeq \plim_{(c \in elF) } {Gc} 
\end{aligned} 
\end{displaymath}
where the canonical limit is indexed by the category $el F$ of elements of $F$.
On the other hand, $\plim(elF \xrightarrow{Q} \mc{C} \xrightarrow{G} \mc{B}) = \plim_{(c \in el F)}{Gc}$. This establishes the isomorphism. 
\end{proof}
A similar result holds for the case of indexed colimits. Let $F:\mc{C}^{op} \rightarrow \fD$ be a smooth index functor and $G:\mc{C}
\rightarrow \mc{B}$ be a smooth functor of diffeological categories. Then the following result holds:
\begin{prop}\label{repcolim2}
\begin{displaymath}
F*G \simeq \colim \left((elF)^{op} \xrightarrow{Q^{op}} \mc{C} \xrightarrow{G} \mc{B}\right)
\end{displaymath}
\end{prop}
\begin{proof}
As before, Theorem \eqref{repcolim} gives a representation of $F$, this time as $F \simeq \colim \mc{C}(-,c)$ since $F$ is a contravariant functor. Therefore 
\begin{displaymath}
\begin{aligned}
F*G &\simeq \left(\colim\mc{C}(-,c)\right)*G\\
&\simeq (\Delta 1*\mc{C}(-,c))*G \\
&\simeq \Delta 1*\left(\mc{C}(-,c)*G \right)\\
&\simeq \Delta 1*(Gc)\\
&\simeq \colim_{(c \in elF)}(Gc)\\
\end{aligned}
\end{displaymath}
The other side of the isomorphism is $\colim\left((elF)^{op} \xrightarrow{Q^{op}} \mc{C} \xrightarrow{G} \mc{B} \right) = \colim_{(c \in elF)}Gc$. The isomorphism is now established.
\end{proof}

\begin{prop}\label{rightkandef}
\begin{displaymath}
\plim \left[(c\downarrow K)\xlongrightarrow{Q} \mc{M} \xlongrightarrow{T} \mc{A} \right] \simeq \{\mc{C}(c,K-),T \}
\end{displaymath}
\end{prop}

\begin{proof}
Corollary \eqref{replim1} gives $\{F,T\} \simeq \plim\left[el F \xlongrightarrow{Q} \mc{M} \xlongrightarrow{T} \mc{A}  \right]$. Defining $F$ as $\mc{C}(c,K-)$ gives $el F = (c\downarrow K)$ as shown above. Hence the result holds. 
\end{proof}
Propositon \eqref{rightkandef} shows that the right Kan extension of the smooth functor $T$, if it exists, is isomorphic to the indexed limit. This leads to the following theorem:

\begin{thm}
The \emph{right Kan extension} of the smooth functor $T: \mc{M} \rightarrow \mc{A}$ along $K: \mc{M} \rightarrow \mc{C}$ is the functor $R: \mc{C} \rightarrow \mc{A}$, given by $Rc = \{\mc{C}(c,K-),T \}$, whenever it exists.
\end{thm}
\begin{proof}
We first construct a smooth natural transformation $\mu: RK \dot{\longrightarrow} T$. Since $Rc = \{\mc{C}(c,K-),T\}$ is the limit of $T$ indexed by $\mc{C}(c,K-)$, equation \eqref{eq:limit} from definition \eqref{indexformulas}  gives a universal arrow $\lambda_{c,-}: \mc{C}(c,K-) \dot{\longrightarrow} \mc{A}(Rc,T-)$. Suppose $c = Km$ for some object $m$ of $\mc{M}$. This gives the transformation
$\lambda_{Km,-}:\mc{C}(Km,K-) \dot{\longrightarrow} \mc{A}(RKm,T-)$. Choosing the second component of $\lambda$ to be $m$ gives
\begin{displaymath}
\lambda_{Km,m}: \mc{C}(Km,Km) \dot{\longrightarrow} \mc{A}(RKm,Tm)
\end{displaymath}
Now define $\varepsilon_{m}: RKm \rightarrow Tm$ as $\varepsilon_{m} = \lambda_{Km,m}(1_{Km})$, where $1_{Km}$ is the identity map in $\mc{C}(Km, Km)$. We need to show that $\varepsilon_{m}$ forms the component of a smooth natural transformation $\varepsilon: RK \dot{\longrightarrow} T$. Let $g: m \rightarrow n$ be an arrow in $\mc{M}$. We have the following commutative diagram:
\begin{displaymath}
\xymatrixcolsep{5pc}\xymatrix{Km \ar[r]\ar[d]_{Kg} & Km\ar[d]^{Kg}\\
Kn\ar[r] & Kn}
\end{displaymath}
from which we deduce clearly that $1_{Kn}\circ Kg = Kg \circ 1_{Km}$. 

Now consider the following commutative diagrams
\begin{equation}\label{commdiag1}
\xymatrixcolsep{5pc}\xymatrixrowsep{4pc}\xymatrix{\mc{C}(Km,Km) \ar[r]^{\lambda_{Km,m}}\ar[d]_{\mc{C}(1,Kg)} & \mc{A}(RKm,Tm)\ar[d]^{\mc{A}(1,Tg)}\\
\mc{C}(Km,Kn) \ar[r]^{\lambda_{Km,n}} & \mc{A}(RKm,Tn)\\
\mc{C}(Kn,Kn) \ar[r]^{\lambda_{Kn,n}}\ar[u]^{\mc{C}(Kg,1)} & \mc{A}(RKn,Tn)\ar[u]_{\mc{A}(RKg,1)}}
\end{equation}
Both the upper and lower squares commute, since $\lambda$ is a universal arrow.
\begin{equation}\label{comdiag2}
\xymatrixcolsep{4pc}\xymatrixrowsep{4pc}\xymatrix{1_{Km}\ar@{|->}[r]\ar@{|->}[d] & \varepsilon_{m}\ar@{|->}[d] \\
Kg \circ 1_{Km}\ar@{|->}[r] & \lambda_{Km,n}(Kg \circ 1_{Km})= Tg \circ \varepsilon_{m} }
\end{equation}
\begin{equation}\label{comdiag3}
\xymatrixcolsep{4pc}\xymatrixrowsep{4pc}\xymatrix{1_{Kn} \circ Kg\ar@{|->}[r] & \lambda_{Km,n}(1_{Kn} \circ Kg)= RKg \circ \varepsilon_{n}\\
1_{Kn}\ar@{|->}[r]\ar@{|->}[u] & \varepsilon_{n}\ar@{|->}[u] }
\end{equation}
The result $1_{Kn}\circ Kg = Kg \circ 1_{Km}$ enables us to conclude from the commutative diagrams that $Tg \circ \varepsilon_{m}=
RKg \circ \varepsilon_{n}$, establishing the natural transformation $\varepsilon: RK \dot{\longrightarrow}T$.

Now suppose $S:\mc{C} \longrightarrow \mc{A}$ is a smooth functor with $\sigma: SK \dot{\longrightarrow} T$ a smooth natural transformation. We show that $\sigma$ uniquely factors through $\varepsilon$ by means of the adjunction in equation \eqref{adjunction1}. We shall use Kelly's approach by employing the calculus of ends and coends. We have:

\begin{displaymath}
\begin{aligned}
Nat(S,R) &= \mc{A^{C}}(S,R) \\
&\simeq \int_{c}\mc{A}(Sc,Rc) \\
&= \int_{c} \mc{A}\left(Sc,\{\mc{C}(c,K-),T \} \right)\\
\end{aligned}
\end{displaymath}
 Applying equation \eqref{eq:limit}  in the case where $F =\mc{C}(c,K-)$ and $G = T$, we have 
 \begin{displaymath}
 \begin{aligned}
 \mc{A}\left(Sc,\{\mc{C}(c,K-),T \} \right)&\simeq  \fD^{\mc{M}}\left[\mc{C}(c,K-),\mc{A}(Sc,T-) \right]\\
\therefore \quad \int_{c} \mc{A}\left(Sc,\{\mc{C}(c,K-),T \} \right)&\simeq \int_{c} \fD^{\mc{M}}\left[\mc{C}(c,K-),\mc{A}(Sc,T-) \right]\\
\text{ Now } \fD^{\mc{M}}\left[\mc{C}(c,K-),\mc{A}(Sc,T-) \right]&= Nat\left(\mc{C}(c,K-),\mc{A}(Sc,T-) \right). \\
\end{aligned}
\end{displaymath}
Writing the above equation as an end, we get
\begin{displaymath}
\begin{aligned}
\fD^{\mc{M}}\left[\mc{C}(c,K-),\mc{A}(Sc,T-) \right]&= \int_{m}\fD\left[\mc{C}(c,Km), \mc{A}(Sc, Tm) \right] \\
\text{So, }\int_{c}\mc{A}\left(Sc,\{\mc{C}(c,K-),T \} \right) &\simeq \int_{c}\int_{m}\fD\left[\mc{C}(c,Km), \mc{A}(Sc, Tm) \right] \\
\end{aligned}
\end{displaymath}
Applying the Fubini Theorem,
\begin{displaymath}
\begin{aligned}
\int_{c}\mc{A}\left(Sc,\{\mc{C}(c,K-),T \} \right) &\simeq \int_{m}\int_{c}\fD\left[\mc{C}(c,Km), \mc{A}(Sc, Tm)  \right]\\
&\simeq \int_{m}\fD^{\mc{C}^{op}}\left[\mc{C}(-,Km),\mc{A}(S-,Tm) \right] \\
&\simeq \int_{m} \mc{A}(SKm, Tm)\\
&\simeq \mc{A^{M}}(SK,T) = Nat(SK,T)\\
\end{aligned}
\end{displaymath} 
The isomorphism $Nat(S,R) \simeq Nat(SK,T)$ gives by Yoneda's Proposition (\cite {Mac}, Proposition $III.2.1$) that the natural transformation $\varepsilon :RK \dot{\longrightarrow} T$ is universal. Hence $(R, \varepsilon)$ is a right Kan extension of the smooth functor $T$.
\end{proof}
 The process of establishing the left Kan extension is dual to the above. Suppose $\mc{A}$, $\mc{C}$ and $\mc{M}$ are diffeological categories with $K: \mc{M} \rightarrow \mc{C}$ and $T: \mc{M} \rightarrow \mc{A}$ smooth functors. Let $F = \mc{C}(K-,c): \mc{M}^{op} \longrightarrow \fD$ be the smooth functor defined by the association $m \longmapsto \mc{C}(Km,c)$. For an arrow $\tau: m \rightarrow
 m^{\prime}$, we have the smooth function $\mc{C}(K\tau,c): \mc{C}(Km^{\prime},c) \longrightarrow \mc{C}(Km,c)$. Hence we have an induced map 
 \begin{displaymath}
 \begin{aligned}
 F^{\ast}: \mc{M}^{op}(m^{\prime},m)&\longrightarrow \fD\left(\mc{C}(Km^{\prime},c), \mc{C}(Km,c) \right)\\
 \left(F^{\ast}\tau^{op}\right)f^{\prime} &= f^{\prime}\circ K\tau\\
 \end{aligned}
 \end{displaymath}
 where $\tau^{op}:m^{\prime} \rightarrow m$ is in a one-to-one correspondence with $\tau:m \rightarrow m^{\prime}$. This gives the commutative diagram
\begin{displaymath}
\xymatrixcolsep{5pc}\xymatrix{Km\ar[rd]_{f}\ar[rr]^{K\tau} & & Km^{\prime}\ar[ld]^{f^{\prime}}\\
& c}
\end{displaymath} 
The opposite category $(el F)^{op}$ of elements of $F$ has as objects $<m,f>$, and arrows $\tau^{\ast}:<m^{\prime},f^{\prime}> \longrightarrow <m,f>$ induced by arrows $\tau: m \rightarrow m^{\prime}$ from the original category $\mc{M}$ such that $\tau^{\ast}f^{\prime} = f$. From the definition of the comma category, $(el F)^{op}$ is exactly the category $(K \downarrow c)$.

\begin{prop}
\begin{displaymath}
\colim \left[(K\downarrow c) \xlongrightarrow{Q^{op}} \mc{M} \xlongrightarrow{T} \mc{A}\right] \simeq \mc{C}(K-,c) * T
\end{displaymath}
\end{prop}
\begin{proof}
From proposition \eqref{repcolim2}, we have for $F = \mc{C}(K-,c)$ and $G=T: \mc{M} \longrightarrow \mc{A}$ the result
\begin{displaymath}
\colim \left[(elF)^{op} \xlongrightarrow{Q^{op}} \xlongrightarrow{T} \right] \simeq F * T
\end{displaymath}
Since $(el F) = (K \downarrow c)$, the result holds. 
\end{proof}
\begin{thm}\label{leftkanexists}
Let $K: \mc{M} \longrightarrow \mc{C}$ and $T: \mc{M} \longrightarrow \mc{A}$ be smooth functors of diffeological categories. The left Kan extension of $T$ along $K$ is given by the functor $L: \mc{C}\rightarrow \mc{A}$, given by $Lc = \mc{C}(K-,c) * T$, whenever it exists.  
\end{thm}
\begin{proof}
From definition \eqref{indexformulas}, there exits a universal arrow $\lambda_{-,c}: \mc{C}(K-,c) \dot{\longrightarrow} \mc{A}(T-,Lc)$. Suppose $c = Km$ for some object $m$ of $\mc{M}$. We obtain the transformation $\lambda_{-,Km}: \mc{C}(K-,Km) \dot{\longrightarrow} \mc{A}(T-,LKm)$. Choosing the first component of $\lambda$ to be $m$, we get
\begin{displaymath}
\lambda_{m,Km}: \mc{C}(Km,Km) \dot{\longrightarrow} \mc{A}(Tm,LKm)
\end{displaymath} 
Define $\varepsilon: T \dot{\longrightarrow} LK$ by $\varepsilon_{m}= \lambda_{m,Km}(1_{Km}): Tm \longrightarrow LKm$. Just like in the  proof for the case of the right Kan extension, we show that this definition gives a universal natural transformation. Let $g:m\rightarrow n$ be an arrow in $\mc{M}$. We have the equality $1_{Kn} \circ Kg = Kg \circ 1_{Km}$. Consider the following commutative diagrams:
\begin{equation}\label{commdiag1}
\xymatrixcolsep{5pc}\xymatrixrowsep{4pc}\xymatrix{\mc{C}(Km,Km) \ar[r]^{\lambda_{m,Km}}\ar[d]_{\mc{C}(1,Kg)} & \mc{A}(Tm,LKm)\ar[d]^{\mc{A}(1,LKg)}\\
\mc{C}(Km,Kn) \ar[r]^{\lambda_{m,Kn}} & \mc{A}(Tm,LKn)\\
\mc{C}(Kn,Kn) \ar[r]^{\lambda_{Kn,n}}\ar[u]^{\mc{C}(Kg,1)} & \mc{A}(Tn,LKn)\ar[u]_{\mc{A}(Tg,1)}}
\end{equation}
Both the upper and lower squares commute, since $\lambda$ is a universal arrow. Splitting the diagram into two components and evaluating the identity map for each component gives
\begin{equation}\label{comdiag2}
\xymatrixcolsep{4pc}\xymatrixrowsep{4pc}\xymatrix{1_{Km}\ar@{|->}[r]\ar@{|->}[d] & \varepsilon_{m}\ar@{|->}[d] \\
Kg \circ 1_{Km}\ar@{|->}[r] & \lambda_{m,Kn}(Kg \circ 1_{Km})= LKg \circ \varepsilon_{m} }
\end{equation}
\begin{equation}\label{comdiag3}
\xymatrixcolsep{4pc}\xymatrixrowsep{4pc}\xymatrix{1_{Kn} \circ Kg\ar@{|->}[r] & \lambda_{m,Kn}(1_{Kn} \circ Kg)= Tg \circ \varepsilon_{n}\\
1_{Kn}\ar@{|->}[r]\ar@{|->}[u] & \varepsilon_{n}\ar@{|->}[u] }
\end{equation}
From the result $1_{Kn}\circ Kg = Kg \circ 1_{Km}$ we get $Tg \circ \varepsilon_{n}=
LKg \circ \varepsilon_{m}$, establishing the natural transformation $\varepsilon: T \dot{\longrightarrow}LK$. This gives the diagram
\begin{displaymath}
\xymatrixcolsep{5pc}\xymatrix{Tm\ar[r]^{\varepsilon_{m}}\ar[d]_{Tg} & LKm\ar[d]^{LKg}\\
Tn\ar[r]_{\varepsilon_{n}} & LKn \\
}
\end{displaymath}
Finally, we show that $\varepsilon: T \dot{\longrightarrow} LK$ is unique up to natural isomorphism. Let $S: \mc{C} \longrightarrow \mc{A}$ be a smooth functor with $\sigma: T \dot{\longrightarrow} SK$ a natural transformation. We have:
\begin{displaymath}
\begin{aligned}
Nat(L,S) &\simeq \mc{A}^{\mc{C}}(L,S)\\
&\simeq \int_{c}\mc{A}(Lc, Sc) \\
&\simeq \int_{c}\mc{A}(\mc{C}(K-,c)*T,Sc)\\
\end{aligned}
\end{displaymath}
by the definition of $Lc$. Applying equation \eqref{repcolim2}, we have
\begin{displaymath}
\begin{aligned}
\mc{A}(\mc{C}(K-,c)*T,Sc) &\simeq \fD^{\mc{M}^{op}}\left[\mc{C}(K-,c),\mc{A}(T-,Sc) \right]\\
\text{ Therefore } \int_{c}\mc{A}(\mc{C}(K-,c)*T,Sc) &\simeq \int_{c}\fD^{\mc{M}^{op}}\left[\mc{C}(K-,c),\mc{A}(T-,Sc) \right]\\
& \simeq \int_{c}\int_{m}\fD\left[\mc{C}(Km,c),\mc{A}(Tm, Sc) \right]\\
\end{aligned}
\end{displaymath}
Applying the Fubini theorem, we get
\begin{displaymath}
\begin{aligned}
\int_{c}\mc{A}(\mc{C}(K-,c)*T,Sc) &\simeq \int_{m}\int_{c}\fD\left[\mc{C}(Km,c),\mc{A}(Tm,Sc) \right]\\
&\simeq \int_{m}\fD^{\mc{C}}\left[\mc{C}(Km,-), \mc{A}(Tm, S-) \right]\\
&\simeq \mc{A}(Tm, SKm)\\
&\simeq \mc{A}^{\mc{M}}(T,SK) = Nat (T, SK)\\
\end{aligned}
\end{displaymath}
Hence by Yoneda's Proposition (\cite{Mac}, Proposition $III.2.1$), the isomorphism $ Nat(L,S) \simeq Nat(T, SK)$ gives the result that the natural transformation $\varepsilon: T \dot{\longrightarrow} LK$ is universal, establishing that $(L,\varepsilon)$ is a left Kan extension of $T$ along $K$.
\end{proof}

\section{The Proof of Theorem \ref{mainthm}}
Let us consider the category $\cC$ of topological (super) Hilbert spaces, where the arrows are continuous linear transformations. We can restrict this to the subcategory where the arrows are maps $T : V \rightarrow W$, with $||T|| \le 1$. As before, define the (super) Lie groups $\mathcal{G}$ and $\mathcal{H}$ to be $\dginf(M)$ and $\sinf(M)$ respectively, where $\mathcal{H}$ is a subgroup of $\mathcal{G}$. We can use the function space diffeology to make $\cC(V,W)$ into a diffeological space (induced from the structures on $V$ and $W$).

Now let $\mc{J}$ be an index category and $F: \mc{J}^{op} \longrightarrow \fD$, $G: \mc{J} \longrightarrow \mc{C}$ be smooth functors.

Let $\alpha: i \rightarrow j$ be an arrow in $\mc{J}$. Fix an object $c$ of $\mc{C}$ and let $\mu: F \dot{\longrightarrow} \mc{C}(G-,c)$ be an $(F,c)$- cylinder over $G$. This gives a commutative diagram of smooth maps
\begin{displaymath}
\xymatrixcolsep{5pc}\xymatrix{Fi \ar[r]^{\mu_{i}}& \mc{C}(Gi,c) \\
Fj \ar[u]^{F\alpha} \ar[r]_{\mu_{j}}& \mc{C}(Gj,c)\ar[u]_{\mc{C}(G\alpha,c)}}
\end{displaymath}
The commutativity of the diagram gives the following relation: for $f_{j}$ in $Fj$ and an element $g_{i}$ in $Gi$,
\begin{equation}
\mu_{i}\left(F\alpha (f_{j})\right) g_{i} = \mu_{j}(f_{j}) G\alpha( g_{i})
\end{equation}
The above diagram gives for each element $f_{i}$ in $F_{i}$ a continuous linear transformation $\mu_{i}(f_{i}): Gi \longrightarrow c$ with norm 
\begin{equation}\label{firstnorm}
||\mu_{i}(f_{i})|| = \sup_{\substack{g_{i}\in Gi\\g_{i}\ne 0}} \frac{||\left(\mu_{i}(f_{i})\right)(g_{i})||}{||g_{i}||} \le 1
\end{equation} 
We have an induced map
\begin{displaymath}
\begin{aligned}
\phi_{i}: Fi \times Gi & \longrightarrow c \\
\phi_{i}(f_{i},g_{i}) &:= \left(\mu_{i}(f_{i})\right)(g_{i})\\ 
\end{aligned}
\end{displaymath}
This map $\phi$ is continuous and linear in the second variable. 
Now consider the ordered pair $(f_{i},g_{i})$ in the cartesian product $Fi \times Gi$. 
We will denote $(f_{i},g_{i})$ by $f_{i}\cdot g_{i}$. 
Define $f_{i}\cdot Gi$ to be the set $\left\{\,f_{i}\cdot g_{i} \,|\, g_{i}\in Gi\, \right\}$. 
Since $f_{i}\cdot Gi$ is in 1-1 correspondence with $Gi$, we transport the Hilbert space structure to $f_{i}\cdot Gi$, making $f_{i}\cdot Gi$ a Hilbert space. Hence for all $g_{i}$ in $Gi$, $||f_{i}\cdot g_{i}|| = ||g_{i}||$. 
Consider the Hilbert space direct sum
\begin{displaymath}
V:= \sum_{i\in \mc{J}} \sum_{f_{i}\in F_{i}} f_{i}\cdot Gi
\end{displaymath}

Now define $R \subseteq {V}$ as the span of the set
\begin{displaymath}
\left\{\,F\alpha (f_{j})\cdot g_{i} - f_{j}\cdot G\alpha( g_{i}) \,|\, i,j \in \mc{J} \right\}
\end{displaymath}
where $\alpha: i \rightarrow j$ is an arrow in $\mc{J}$ as defined previously.
The subspace $R$ is well defined, since $F\alpha(f_{j})$ is in $Fi$ and $G\alpha(g_{i})$ is in $Gj$.  

\begin{thm}
The orthogonal complement of $R$, $R^{\bot} = F \ast G$.
\end{thm}

\begin{proof}
By the definition of $R$, the orthogonal projection of  
${F\alpha (f_{j})\cdot g_{i}} - {f_{j}\cdot G\alpha( g_{i})}$ gives the zero vector in $R^{\bot}$. 
We define the universal cylinder $\lambda: F \dot{\longrightarrow} \mc{C}(G-, R^{\bot})$ by means of the projection $Fi \times Gi \hookrightarrow V \xlongrightarrow{\pi} R^{\bot}$ , where the projection is given by the composition
\begin{displaymath}
f_{i}\cdot g_{i} \longmapsto \sum f_{i}\cdot g_{i} \longrightarrow \pi(f_{i}\cdot g_{i})
\end{displaymath}
where the sum is taken over all $i$ in $\mc{J}$ and $g_{i}$ in $G_{i}$. This gives $\lambda_{i}: Fi \rightarrow \mc{C}(Gi, R^{\bot})$ as follows: fix an element $f_{i}$ of $Fi$. We have the map
 \begin{displaymath}
 \begin{aligned}
 \lambda_{i}(f_{i}): Gi &\longrightarrow R^{\bot}\\
 \left(\lambda_{i}(f_{i})\right)g_{i} &= \pi(f_{i}\cdot g_{i}) \\
 \end{aligned}
 \end{displaymath}
 We show that $\lambda_{i}(f_{i})$ is an arrow of $\mc{C}$. The norm of $\lambda_{i}(f_{i})$ is given by
 \begin{displaymath}
 ||\lambda_{i}(f_{i})|| = \sup_{\substack{g_{i}\in Gi\\ g_{i}\ne0}} \frac{||\pi(f_{i}\cdot g_{i})||}{||g_{i}||} \le \sup_{\substack{g_{i}\in Gi\\ g_{i}\ne0}} \frac{||f_{i}\cdot g_{i}||}{||g_{i}||} = 1
 \end{displaymath}
 since $\pi$ is an orthogonal projection and therefore has norm $||\pi|| \le 1$ and $||f_{i}\cdot g_{i}|| = ||g_{i}||$.
 

The naturality of $\lambda$ is given by the following diagram: 
\begin{displaymath}
\xymatrixcolsep{5pc}\xymatrix{Fi\ar[r]^{\lambda_{i}} & \mc{C}(Gi, R^{\bot}) \\
Fj\ar[u]^{F\alpha}\ar[r]_{\lambda_{j}} & \mc{C}(Gj, R^{\bot})\ar[u]_{\mc{C}(G\alpha,R^{\bot})}}
\end{displaymath}
All the maps in the diagram are smooth maps of diffeological spaces. 
We show that the diagram commutes. Let $f_{j}$ be an element of the diffeological space $Fj$. We have:
\begin{displaymath}
\begin{aligned}
\lambda_{i}(F\alpha(f_{j}))g_{i} &= \pi(F\alpha(f_{j})\cdot g_{i})\\
	&= \pi(f_{j}\cdot G\alpha(g_{i})) \quad \text{by the definition of }R\\
	&= \lambda_{j}(f_{i})(G\alpha(g_{i})). 
\end{aligned}
\end{displaymath} 
Therefore the diagram commutes and $\lambda: F \dot{\longrightarrow} \mc{C}(G-,R^{\bot})$ is a cylinder. 
To show that $\lambda: F\dot{\longrightarrow} \mc{C}(G-, R^{\bot})$ is a universal cylinder, suppose we have another $(F,c)-$cylinder $\beta: F \dot{\longrightarrow} (G-,c)$. This gives the commutative diagram illustrated below.

\begin{equation}\label{newcylinder}
\xymatrixcolsep{5pc}\xymatrix{Fi\ar[r]^{\beta_{i}} & \mc{C}(Gi,c)\\
Fj\ar[u]^{F\alpha}\ar[r]_{\beta_{j}} & \mc{C}(Gj,c)\ar[u]_{\mc{C}(G\alpha,c)}}
\end{equation}
Therefore for each $f_{j}$ in $Fj$, $\beta_{i}(F\alpha (f_{j}))(g_{i})=\beta_{i}(f_{j})(G\alpha (g_{i}))$.

The unique map $T: R^{\bot} \rightarrow c$ is defined as $T\left[\pi\left(f_{i}\cdot g_{i}\right) \right] = \left(\beta_{i}(f_{i})\right)(g_{i})$. We show that $T$ is well defined by proving that the generators of $R$ under the action of $T$ give the zero vector in $c$. Pick a generator $F\alpha (f_{j})\cdot g_{i} - f_{j}\cdot G\alpha( g_{i})$ in $R$. Then $\pi(F\alpha (f_{j})\cdot g_{i} - f_{j}\cdot G\alpha( g_{i})) = \pi(F\alpha (f_{j})\cdot g_{i}) - \pi({f_{j}\cdot G\alpha( g_{i})}) = 0$ in $R^{\bot}$. We have 
\begin{displaymath}
\begin{aligned} 
T\left[\pi(F\alpha (f_{j})\cdot g_{i}) - \pi({f_{j}\cdot G\alpha( g_{i})}) \right] &= \beta_{i}(F\alpha (f_{j}))(g_{i})-\beta_{i}(f_{j})(G\alpha (g_{i}))\\
& = 0,
\end{aligned}
\end{displaymath}
by the commutative diagram \eqref{newcylinder}.
Hence for all pairs $(f_{i},g_{i})$ such that (by means of the inclusion map) $f_{i}\cdot g_{i}$ is in  $R$, $(\beta_{i}f_{i})= 0$ in $c$.

Now we compute the norm of $T$. We have the following commutative diagrams:
\begin{displaymath}
\xymatrixcolsep{5pc}\xymatrix{Fi \times Gi\ar[d]_{\pi}  \ar[r] & c\\
R^{\bot}\ar@{-->}[ru]_{T}}
\xymatrixcolsep{5pc}\xymatrix{(f_{i},g_{i})\ar@{|->}[d]  \ar@{|->}[r] & \beta_{i}(f_{i})g_{i}\\
\pi(f_{i}\cdot g_{i})\ar@{|-->}[ru]}
\end{displaymath}
We have
\begin{equation}\label{normT}
||T|| = \sup_{\substack{g_{i} \in Gi\\g_{i}\ne 0}}\frac{||\beta_{i}(f_{i})g_{i}||}{||\pi(f_{i}\cdot g_{i})||}
\end{equation}
Fix a pair $(f_{i},g_{i})$ in $Fi \times Gi$. We have that $||\pi(f_{i}\cdot g_{i})|| \le ||f_{i}\cdot g_{i}||$. If $||\pi(f_{i}\cdot g_{i})|| = ||f_{i}\cdot g_{i}||$, then we have 

\begin{displaymath}
\frac{||\beta_{i}(f_{i})g_{i}||}{||\pi(f_{i}\cdot g_{i})||}=\frac{||\beta_{i}(f_{i})g_{i}||}{||f_{i}\cdot g_{i}||}\le \frac{||g_{i}||}{||f_{i}\cdot g_{i}||} = 1.
\end{displaymath}
Now if $||\pi(f_{i}\cdot g_{i})|| < ||f_{i}\cdot g_{i}||$, then we can decompose $f_{i}\cdot g_{i}$ in $V$ into the components in $R$ and $R^{\bot}$ to get $f_{i}\cdot g_{i} = f_{i}\cdot g_{i}^{\prime} + f_{i}\cdot g_{i}^{\prime\prime}$, where $f_{i}\cdot g_{i}^{\prime}$ is in $R$ and $f_{i}\cdot g_{i}^{\prime\prime}$ is in $R^{\bot}$. We can express $f_{i}\cdot g_{i}^{\prime}$ as $f_{i}\cdot g_{i}^{\prime} = F\alpha (h_{j})\cdot g_{i}^{\prime} - h_{j} \cdot G\alpha(g_{i}^{\prime})$, where $h_{j}$ is in $F_{j}$ and $\alpha$ is an arrow in $\mc{J}$. Since $\beta$ is a $(F,c)-$cylinder under $G$, $\beta_{i}(f_{i})g_{i} = \beta_{i}\left[F\alpha (h_{j})\cdot g_{i}^{\prime} - h_{j} \cdot G\alpha(g_{i}^{\prime}) \right]=0$. Hence we have the following results:
\begin{displaymath}
\begin{aligned}
\text{In } R^{\bot},\quad ||\pi(f_{i}\cdot g_{i})|| &=||f_{i}\cdot g_{i}^{\prime\prime}|| = ||g_{i}^{\prime\prime}||\\
\text{In } c,\quad ||\beta_{i}(f_{i})g_{i}||&= ||\beta_{i}(f_{i})g_{i}^{\prime\prime}||\le ||g_{i}^{\prime\prime}||\\
\therefore\quad \frac{||\beta_{i}(f_{i})g_{i}||}{||\pi(f_{i}\cdot g_{i})||}&= \frac{||\beta_{i}(f_{i})g_{i}^{\prime\prime}||}{||f_{i}\cdot g_{i}^{\prime\prime}||} \le \frac{||g_{i}^{\prime\prime}||}{||g_{i}^{\prime\prime}||} = 1
\end{aligned}
\end{displaymath}
Therefore, in both cases we have   
\begin{displaymath}
||T|| = \sup_{\substack{g_{i} \in Gi\\g_{i}\ne 0}}\frac{||\beta_{i}(f_{i})g_{i}||}{||\pi(f_{i}\cdot g_{i})||} \le 1
\end{displaymath}
Hence T is an arrow in $\mc{C}$ and
$\lambda$, as defined is a universal cylinder.
We conclude that $R^{\bot} = F \ast G$.
\end{proof}
\emph{Proof of Theorem \ref{mainthm}}

Given a representation $T: \mathcal{H} \rightarrow \cC$, define the inclusion functor $I : \mathcal{H} \rightarrow \mathcal{G}$. The functor $T : \cH \rightarrow \cC$ is a representation of $\cH$. The left Kan extension $L : \cG \rightarrow C$ takes the form $L(\cG) = \textrm{Hom}(I-,G) \ast T)$, which exists by the proof of the previous theorem, and the completeness of the category $\fD$. Given a representation $S: \cG \rightarrow \cC$, the composite $SI$ is the restriction of the representation to $\cH$. 

The adjunction can therefore be expressed $Nat(L,S)  \simeq Nat(T, SI)$. Let $V = T(\cH)$ be an $\cH$-module, and $W$ a $\cG$-module. Denoting $L(G)$ by $I\cdot V$, we obtain the isomorphism (of diffeological spaces)  
\[
\textrm{Hom}_{\cG}(I\cdot V, W) \simeq \textrm{Hom}_{\cH}(V, W)
\]
establishing that $I\cdot V$ is the (co) induced representation of $\cG$. 

\section{Remarks}

\begin{rem}
The definition of diffeological category raises a question; are there examples of categories with the objects not necessarily diffeological spaces, such that the hom-sets have the structure of a diffeological space? It seems that the underlying structures would need to have the structure of a diffeological space, so we can use the function space diffeology. This means that the category $\fD$ is its own higher-category: for $G, H$ diffeological spaces, the morphism $\fD(G,H)$ is also a diffeological space. We can extend this to $2$-morphisms and so on. 
\end{rem}

\begin{rem}
Milnor discusses the issue of regularity of (infinite dimensional) Lie groups in \cite{Mil}. Essentially, a Lie group is regular if there exists an exponential map from the Lie algebra to the Lie group, relating representations of both the algebra and the group. In this paper, we have discussed diffeological Lie groups, and it is natural to ask whether they are regular. J-P. Magnot \cite{Mag} has shown that this is not the case, so not all diffeological Lie groups are regular. This leads to the question whether all regular diffeological Lie groups can be classified. 
\end{rem}

\bibliographystyle{alpha}
\bibliography{refs}

\end{document}